\documentclass{article}

\usepackage[english]{babel}

\usepackage[letterpaper,top=2cm,bottom=2cm,left=3cm,right=3cm,marginparwidth=1.75cm]{geometry}

\usepackage{graphicx}
\usepackage[colorlinks=true, allcolors=blue]{hyperref}
\usepackage{authblk}
\usepackage[utf8]{inputenc} 
\usepackage[T1]{fontenc}    

\usepackage{url}            
\usepackage{booktabs}       
\usepackage{amsfonts,amsmath,amssymb,amsthm}

\usepackage{cleveref}
\usepackage{nicefrac}       
\usepackage{microtype} 
\usepackage{subfigure}
\usepackage{subcaption}
\usepackage{xcolor}         
\usepackage[linesnumbered,ruled,vlined]{algorithm2e}
\usepackage{algorithmic}
\usepackage{mathtools}
\usepackage{natbib}

\newtheorem{theorem}{Theorem}
\newtheorem{remark}{Remark}
\newtheorem{lemma}{Lemma}
\newtheorem{proposition}{Proposition}

\usepackage{array}
\usepackage{adjustbox,tabularx}

\begin{document}

\title{Orientation Determination of Cryo-EM Images Using Block Stochastic Riemannian Subgradient Methods}

\author[1]{Wanyu Zhang\thanks{wanyuzhang@stu.sufe.edu.cn}}
\author[2]{Ruili Gou}
\author[3]{Huikang Liu}
\author[2]{Zhiguo Wang}
\author[4]{Yinyu Ye}
\affil[1]{Shanghai University of Finance and Economics}
\affil[2]{Sichuan University}
\affil[3]{Shanghai Jiao Tong University}
\affil[4]{Stanford University}

\maketitle

\begin{abstract}
The determination of molecular orientations is crucial for the three-dimensional reconstruction of Cryo-EM images. Traditionally addressed using the common-line method, this challenge is reformulated as a self-consistency error minimization problem constrained to rotation groups. In this paper, we consider the least-squared deviation (LUD) formulation and employ a Riemannian subgradient method to effectively solve the orientation determination problem. To enhance computational efficiency, a block stochastic version of the method is proposed, and its convergence properties are rigorously established. Extensive numerical evaluations reveal that our method not only achieves accuracy comparable to that of state-of-the-art methods but also delivers an average 20-fold speedup. Additionally, we implement a modified formulation and algorithm specifically designed to address scenarios characterized by very low SNR.
\end{abstract}

\section{Introduction}
Cryogenic electron microscopy (Cryo-EM) is crucial for reconstructing the 3D structures of macromolecules, providing evidence for the architecture of biomolecules \citep{frank2001cryo}. In Cryo-EM, identical biomolecules are rapidly frozen in their native states within vitreous ice, allowing for the generation of images from various orientations that are integrated to resolve the structure of the observed particles \citep{weissenberger2021understanding}. In the imaging process, only a small amount of electrons is used to reduce damage to molecules, resulting in images with a low signal-to-noise ratio (SNR).

A prerequisite for 3D reconstruction is to obtain the projection direction of each image. However, this task is challenging because Cryo-EM projection directions are random and unknown \cite{bendory2020single, lyumkis2019challenges}. This randomness arises from the unpredictable orientations and positions of biomolecules during sample preparation and the randomizing effect of the electron beam during imaging. Furthermore, the high-energy electron beam can damage the sample, preventing repeated imaging from known directions. These challenges necessitate sophisticated computational methods to determine directions from randomly oriented noisy images.

In Cryo-EM orientation determination, two streams of methods are used: the random-conical tilt technique \citep{radermacher1986new} and common-line-based algorithms \citep{singer2010detecting,van1987angular,singer2011three,wang2013orientation}. The random-conical tilt method collects images from particles at known and zero tilt angles, directly revealing particle orientations relative to the tilt axis. In contrast, common-line-based algorithms apply the Fourier projection slice theorem \citep{gustafsson1996mathematics}, which posits that the 2D Fourier transform of a projection image forms a slice through the 3D Fourier space orthogonal to the projection direction. This principle shows that two such transformed images intersect along a line—the common line—allowing for the inference of relative orientations by analyzing the similarity of Fourier coefficients along these intersections.

The common-line approach is preferred as it uses the Fourier transform's properties to deduce spatial relationships without physically tilting the samples, thereby reducing radiation risk and avoiding sample damage. It also aids in 3D structural reconstruction by aligning multiple 2D images through their Fourier space intersections \cite{crowther1970three,rosenthal2015high}. Therefore, this paper focuses on the common-line-based approach. 

\textbf{Related work.} In the field of common-line-based orientation determination algorithms, \cite{singer2011three} derive least squares (LS) formulation on the rotation group, and address it with eigenvector relaxation and semi-definite programming (SDP) techniques. \cite{shkolnisky2012viewing} further formulates the problem as a synchronization problem based on LS. To handle outliers from misdetected common-lines, \cite{wang2013orientation} propose a robust least unsquared deviation (LUD) formulation, solved using semidefinite relaxation technique, and apply a spectral norm constraint to manage solution clustering under high noise. \cite{bandeira2020non} view this task as a special case of a non-unique game, which is relaxed by SDP and solved by ADMM. These approaches, however, focus on relaxed versions of the problem. Furthermore, SDP is notoriously difficult to optimize for large-scale data. Recently, \cite{pan2023orientation} solves the LUD formulation directly by transforming the LUD problem into a series of weighted LS problems, tackled through iterative reweighted least squares (IRLS) and projection gradient descent (PGD). 

\textbf{Motivations.} However, most existing methods are computationally inefficient, especially when millions of images are involved in real-world applications. Additionally, these methods struggle to maintain high accuracy as SNR decreases. These challenges persist as open problems in the field of orientation determination, motivating us to design more efficient algorithms. Notably, the LUD formulation, as described by \eqref{formula: LUD}, presents a nonconvex and nonsmooth optimization problem on the Riemannian manifold. Extensive research has been conducted on such problems, including works by \citep{absil2008optimization, yang2014optimality,li2021weakly, huang2022riemannian, chen2020proximal}. Specifically, \cite{liu2023resync} has demonstrated the effectiveness of a Riemannian subgradient-based algorithm (ReSync) for robust least-unsquared rotation synchronization, achieving linear convergence. This suggests potential for efficient problem-solving that could be further enhanced by integrating stochastic techniques.

\textbf{Main contributions.} In this paper, we build on the Riemannian subgradient-based algorithm (ReSync) framework introduced by \cite{liu2023resync}, proposing a block stochastic Riemannian subgradient (BSGD) method to address the LUD formulation of the orientation determination problem. Our contributions are summarized as follows:

\begin{itemize}
    \item We propose a block stochastic Riemannian subgradient (BSGD) method to solve the LUD formulation of the orientation determination problem and establish its convergence property. 
    \item Extensive experiments demonstrate that the proposed method achieves comparable accuracy and an averaging 20-fold speedup compared to existing common-line-based state-of-the-art methods.
    \item To tackle the scenarios with very low SNR, we adopt the LUD formulation with a spectral norm constraint and modify our method to effectively solve it.
\end{itemize}

\section{Problem formulation}

Now we present a mathematical formulation of the orientation determination problem. Suppose we have collected $K$ projection images $P_1,...,P_K$ of the same molecule, each of them is taken at randomly unknown directions characterized by some rotation matrices $R_1,...,R_K \in \text{SO}(3)$. Our goal is to recover all the rotation matrices given these noisy projections. 

In this paper, we adopt a common-line-based approach to solve this problem. Firstly, 2D Fourier transformation is applied to projection images \cite{dutt1993fast,fessler2003nonuniform} and pairs of common-lines are subsequently detected. For readers interested in the detailed background of common-line method, refer to \cite{wang2013orientation}. Denote \(c_{ij},c_{ji} \in \mathbb{R}^3\) as the detected common lines in transformed projections $\hat P_i,\hat P_j$, respectively. Since $\hat P_i,\hat P_j$ are inherently obtained by rotations $R_i,R_j$, by Fourier projection-slice theorem, the common-lines would remain to intersect at 3D Fourier space after rotated by $R_i,R_j$. Mathematically, such a relationship can be captured by
\begin{align} \label{eq:cl}
    R_i c_{ij} = R_j c_{ji}, \quad \forall~ 1\leq i < j \leq K.
\end{align} 
So far we have established the basic ideas of common-line-based approaches, then we consider formulating the orientation determination problem as an optimization problem on Stiefel manifolds. Specifically, we present two kinds of formulations, least square (LS) formulation and least unsquared deviation (LUD) formulation.

\textbf{LS formulation.} The overdetermined linear system \eqref{eq:cl} is typically solved by the least square approach, which is formulated by the following manifold optimization problem
\begin{align}\label{formula: LS}
\min_{R_1,...,R_K \in \textbf{SO}(3)} \sum_{i,j\in [K]} \|R_i c_{ij}-R_j c_{ji}\|_2^2.
\end{align}
Problem \eqref{formula: LS} is a nonconvex optimization problem on the Riemannian manifold, favored for its smoothness properties. Previous efforts include \cite{singer2011three} addressing this formulation through eigenvector and semi-definite relaxations. However, the least squares formulation lacks robustness as the sum of squared errors may be heavily influenced by outliers.

\textbf{LUD formulation.} A more robust alternative is the least unsquared deviation formulation, which aims to minimize the sum of unsquared errors:
\begin{align}\label{formula: LUD}
    \min_{R_1,...,R_K \in \textbf{SO}(3)} f(R)=\sum_{i,j\in [K]} \|R_i c_{ij}-R_j c_{ji}\|_2.
\end{align}
Problem \eqref{formula: LUD} is nonconvex and nonsmooth, and is more difficult to optimize than the LS formulation. \cite{wang2013orientation} firstly propose semi-definite relaxation for LUD and apply ADMM algorithm to solve it. \cite{pan2023orientation} applies the projection gradient descent method to solve it. 

\section{Proposed algorithms} \label{sec:algo}

\subsection{Initialization}

We adopt the eigenvector relaxation method for efficient and high-quality initialization. Considering the LS formulation \eqref{formula: LS}, it is equivalent to the maximization problem of the sum of dot products:
\begin{align}\label{init: dot}
    & \max_{R_1,...,R_K \in \textbf{SO}(3)} \sum_{i,j\in [K]} R_i c_{ij} \cdot R_j c_{ji}=\sum_{i,j\in [K]} \text{tr}(R_i c_{ij}c_{ji}^T R_j^T)=\text{tr}(R^T M R) 
\end{align}
where $M_{ij}=c_{ij}c_{ji}^T\in \mathbb{R}^{3\times 3}$, and $M=[M_{ij}]\in\mathbb{R}^{3K\times 3K}$. To make this optimization problem tractable, we relax the manifold constraints and consider the following problem:
\begin{align}\label{init: relaxed}
\begin{split}
     \max_{R\in \mathbb{R}^{3K \times 3}} & \quad \text{tr}(R^T M R) \quad
    \text{subject to} \quad   R^T R = K\cdot I_3.
\end{split}
\end{align}
The solution to the maximization problem \eqref{init: relaxed} is therefore given by the top three eigenvectors $v_1, v_2, v_3$ (multiplied by constant $\sqrt{K}$) of $M$. Let $A=[v_1, v_2, v_3]\in \mathbb{R}^{3K\times 3}$ and it is then divided into $3\times 3$ matrices $A_1,...,A_K$. Subsequently, rotations are recovered by letting $R_i=U_i V_i^T$, where $A_i = U_i \Sigma V_i^T$ is its singular value decomposition.

\subsection{Riemannian subgradient method}
We consider applying the Riemannian subgradient method to solve problem \eqref{formula: LUD}. The main idea is to compute the subgradient in Euclidean space, then the Riemannian subgradient is derived from the projection on tangent space, and finally, the retraction operator is posed to guarantee feasibility. The pseudocode is presented by Algorithm \ref{algo:ReSync}. Here we start by deriving the form of Euclidean subdifferential:
\begin{align}\label{eq:full-subgrad}
    \partial f(R_i)= \sum_{j\in [K]} \partial f_{ij}(R_i)= \begin{dcases*}
  \sum_{j\in [K]} \frac{(R_i c_{ij}-R_j c_{ji})c_{ij}^T}{\|R_i c_{ij}-R_j c_{ji}\|} & if $\|R_i c_{ij}-R_j c_{ji}\| \neq 0$ \\
  V\in \mathbb{R}^{3\times 3}, \|V\|_F\leq 1& \text{otherwise}
\end{dcases*} 
\end{align}
where $f_{ij}=\|R_i c_{ij}-R_j c_{ji}\|_2$. Let $\tilde \nabla f(R_i)\in \partial f(R_i)$ refer to the Euclidean subgradient of block variable $R_i$, then the Riemannian subgradient $\tilde \nabla_{\mathcal{R}} f(R_i)$ can be obtained by project Euclidean subgradient onto tangent space:
\begin{align}\label{eq:proj}
    \tilde \nabla_{\mathcal{R}} f(R_i) = \mathcal{P}_{\text{T}_{R_i}}(\tilde \nabla f(R_i)), 1\leq i \leq K
\end{align}
where the projection on the tangent space $\text{T}_{R_i}$ at $R_i$ can be computed as $\mathcal{P}_{\text{T}_{R_i}}(B)=R_i(R_i^T B - B^T R_i)/2$ for any $B \in \mathbb{R}^{3\times 3}$.
Subsequently, the Riemannian subgradient descent is given by $R_i^+=R_i-\mu \tilde \nabla_{\mathcal{R}} f(R_i)$. Unfortunately, $R_i^+$ may violate the manifold constraint. One common approach is to employ a retraction operator to guarantee feasibility. For $\textbf{SO}(3)$, we can use a QR decomposition-based retraction and implement the Riemannian subgradient step as
\begin{align}\label{eq:retr}
    R_i^+=\text{Retr}_{R_i} (-\mu \tilde \nabla_{\mathcal{R}} f(R_i))=\text{Qr}(R_i-\mu \tilde \nabla_{\mathcal{R}} f(R_i)),
\end{align}
where $\text{Qr}(B)$ refers to the Q-factor in the thin QR decomposition of $B$.

\begin{minipage}{0.46\textwidth}
\begin{algorithm}[H] \label{algo:ReSync}
\SetKwInput{KwInput}{Input}
\SetKwInput{KwOutput}{Output}
\SetKwInput{KwInitialize}{Initialize}
\SetKwInput{KwRequire}{Require}
\SetAlgoLined
\DontPrintSemicolon
\caption{ReSync}
\KwInput{Initialization \( R^0\), stepsize \(\mu_0\)}
Set iteration count \( t = 0 \);\;
\While{stopping criterion not met}{
    Update the step size \( \mu_t \)\;
    \For{\( i = 1 \) \KwTo \( K \)}{
        \(\partial f(R_i)=\sum_{j \in [K]} \partial f_{i,j}(R_i)\)\;
        \(\tilde \nabla_{\mathcal{R}} f(R_i) = \mathcal{P}_{\text{T}_{R_i}}(\tilde \nabla f(R_i))\)\;
        \( R_{i}^{t+1} = \text{Retr}_{R_{i}^t} ( -\mu_k \widetilde{\nabla} f(R_i^t) ) \)\;
    }
    Update iteration count \( t = t + 1 \);\;
}
\end{algorithm}
\end{minipage}
\hfill
\begin{minipage}{0.48\textwidth}
\begin{algorithm}[H]\label{alg:stochastic}
\SetKwInput{KwInput}{Input}
\SetKwInput{KwOutput}{Output}
\SetKwInput{KwInitialize}{Initialize}
\SetKwInput{KwRequire}{Require}
\SetAlgoLined
\DontPrintSemicolon
\caption{Stochastic ReSync}
\KwInput{Initialization \(R^0\), stepsize \(\mu_0\), filter ratio \(\rho_1, \rho_2\)}
Set iteration count \( t = 0 \)\;
\While{stopping criterion not met}{
    Update the step size \( \mu_t \)\;
    Draw \(\mathcal{S},\mathcal{D}\subset [K]\) by \(\rho_1, \rho_2\) \;
    \For{\( i \in \mathcal{D} \)}{
        \(\partial f(R_i)=\sum_{j \in \mathcal{S}} \partial f_{i,j}(R_i)\)\;
        execute line 6-7 in Algorithm \ref{algo:ReSync}\;
    }
    Update iteration count \( t = t + 1 \);\;
}
\end{algorithm}
\end{minipage}

\subsection{Stochastic ReSync algorithms}

In this section, we explore stochastic variants of the ReSync algorithm to enhance its computational speed. We introduce a stochastic update approach where, instead of updating all rotation components using all common lines, we update only a subset of rotations, $\mathcal{D} \subseteq [K], |D|=\rho_1 K$, using a selected subset of data, $\mathcal{S} \subseteq [K], |S|=\rho_2 K$. Specifically, the update rule \eqref{eq:full-subgrad} is substituted by
\begin{align}\label{eq:BSGD-subgrad}
    \partial f(R_i)=\sum_{j \in \mathcal{S}} \partial f_{i,j}(R_i), \quad \forall i\in \mathcal{D}.
\end{align}
We remark that many stochastic algorithms can be reduced to \eqref{eq:BSGD-subgrad}, like stochastic gradient descent (SGD) \cite{amari1993backpropagation}, block coordinate descent (BCD) \cite{tseng2001convergence}, and block stochastic gradient descent (BSGD) \cite{xu2015block, zhao2014accelerated}.
\begin{itemize}
    \item ReSync-SGD: \(\rho_1=1, \rho_2=\rho\), all rotations are updated using $\rho$-fraction of data.
    \item ReSync-BCD: \(\rho_1=\rho, \rho_2=1\), $\rho$-fraction of rotations is updated by full subdifferential.
    \item ReSync-BSGD:  \(\rho_1=\rho_2=\rho\) and \(\mathcal{S}=\mathcal{D}\), only $\rho$-fraction of rotations is updated using the same fraction of data.
\end{itemize}
We refer to the parameter $\rho$ as the filter ratio, which dictates the proportion of data utilized in these updates. After computing the Euclidean subgradient, we apply projection and retraction operators to a subset $\mathcal{D}$ of rotations. The pseudocode for the general stochastic ReSync algorithm is outlined in Algorithm \ref{alg:stochastic}.

Additionally, we can enhance the stability of stochastic algorithms by incorporating the random reshuffling (RR) trick \cite{safran2020good,li2023convergence}. It is also noteworthy that BSGD offers an intuitive method for solving orientation determination problems. Given that the objective function \eqref{formula: LUD} features a double-sum structure and symmetry, in each iteration, we simply sample a subset of rotations and optimize the corresponding subproblem:
\begin{align} \label{formula: LUD-sub}
& \min_{R_i \in \textbf{SO}(3), i\in \mathcal{S}} \sum_{i, j \in \mathcal{S}} \|R_i c_{ij}-R_j c_{ji}\|_2.
\end{align}
Another advantage of BSGD is its capacity for parallel execution. In each iteration, the rotations can be divided into separate blocks, allowing for the simultaneous optimization of independent subproblems. This underscores BSGD's computational benefits for large-scale applications. In contrast, such parallelization is not feasible with BCD, as it employs a Gauss-Seidel-style update that necessitates sequential execution.

\subsection{Spectral norm constraints for low SNR} \label{sec:norm}

\begin{figure}
    \centering
    \subfigure[Ground-truth]{\includegraphics[width=0.2\textwidth]{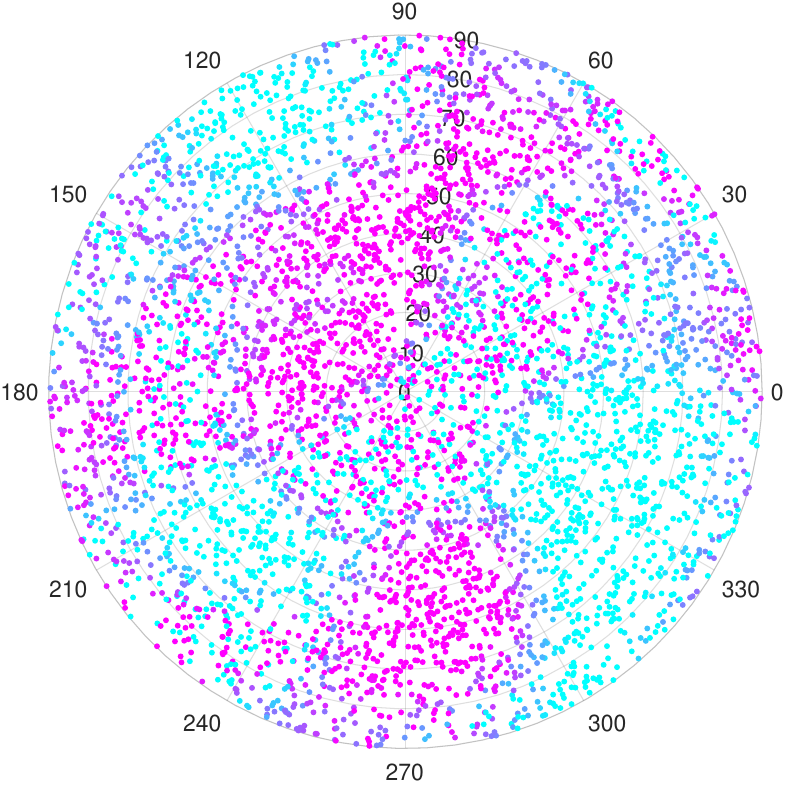}}
    \hspace{1 cm}
    \subfigure[ReSync]{\includegraphics[width=0.2\textwidth]{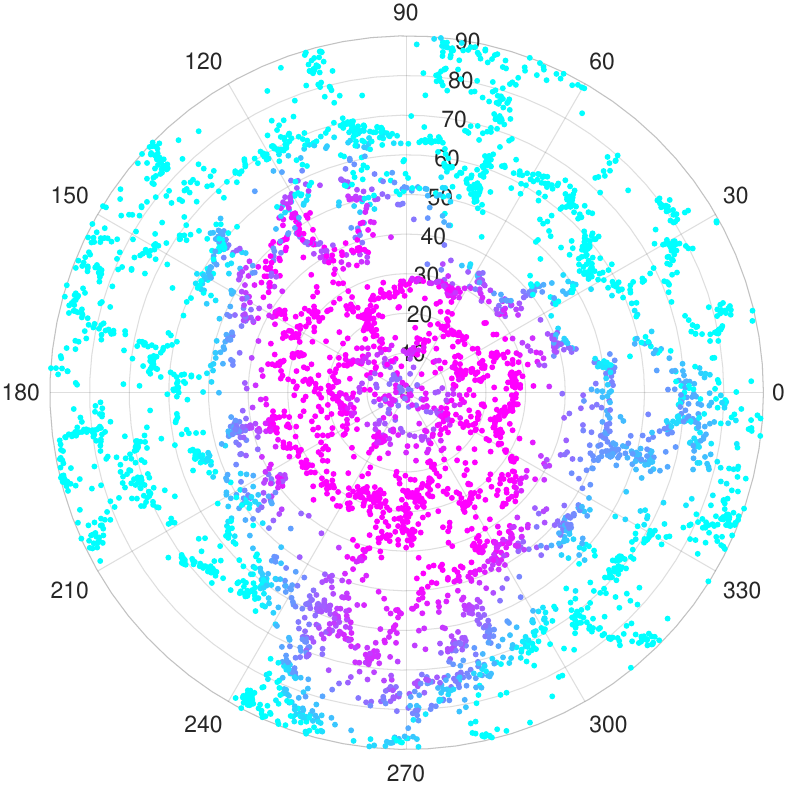}}
    \hspace{1 cm}
    \subfigure[ReSync-norm]{\includegraphics[width=0.23\textwidth]{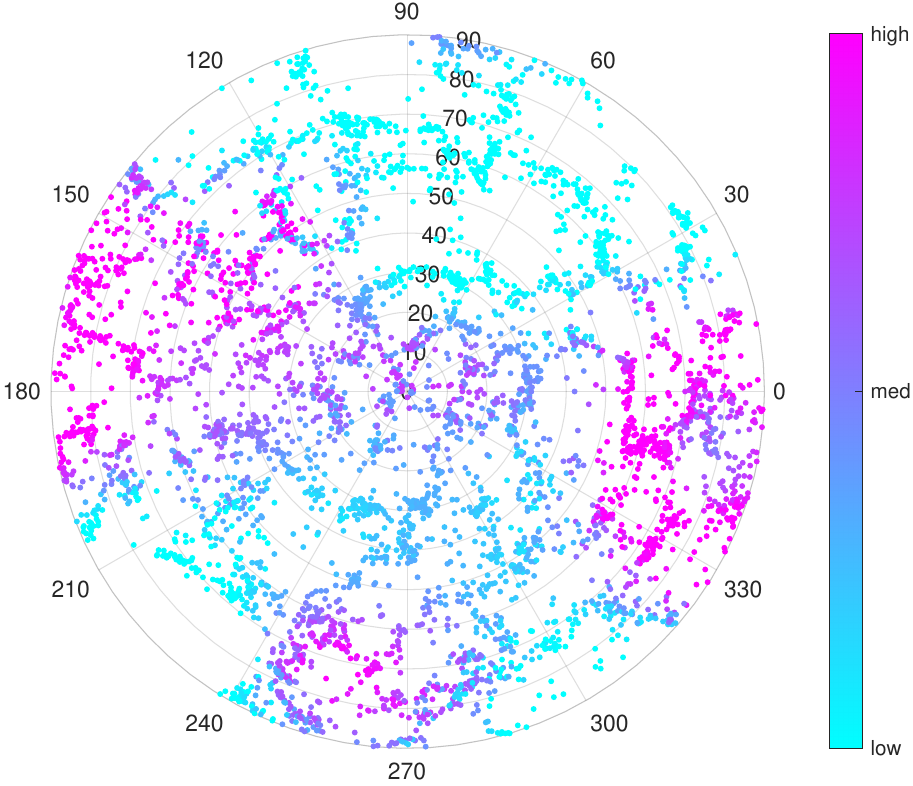}}
    \caption{2D visualization of viewing directions by ground-truth and ReSync algorithms ($K=1000,SNR=1/64$). Each point represents a 3D viewing direction. The highly clustered area is in purple, while low-density area is in blue.}
    \label{fig:vis-direction}
    \vspace{-0.5cm}
\end{figure}

When the images are extremely noisy, there is an empirical observation that most algorithms fail to find high-accuracy solutions. The explanation for this observation is that, in this case, the solved viewing directions are highly clustered at some small region and thus it is far from uniformly distributed ground truth \cite{wang2013orientation}. We also provide visualization of projection direction exhibited in \Cref{fig:vis-direction} (a) (b). 

Denote by $\widetilde{R} \in \mathbb{R}^{2K \times 2}$ the matrix consists of the first two rows of each $R_i, 1 \leq i \leq K$. As suggested in \cite{wang2013orientation}, in this case, the Gram matrix $G=\widetilde{R} \widetilde{R}^T$ has just two dominant eigenvalues, instead of three. Therefore, to prevent clustering, spectral norm constraint \(\|G\|_2 \preceq \alpha K\) is added to the original formulation \eqref{formula: LS} and \eqref{formula: LUD} where \(\alpha = [\frac{2}{3},1) \) is a hyperparameter. Besides, if these rotations are uniformly sampled, then \(\alpha\) is close to \(\frac{2}{3}\). \cite{wang2013orientation} apply this constraint on semi-definite relaxation of LUD formulation, and use alternating direction method of multipliers (ADMM) to solve this problem, referred to as SDR-ADMM method hereafter. 

However, SDR is always time-consuming, so we incorporate spectral norm constraints into our algorithm. We add a surrogate constraint \(\widetilde{R}^T \widetilde{R} = \alpha K I_3\) to get the following optimization problem 
\begin{equation}\label{formula: LUD-norm}
\begin{aligned}
    & \min_{R_1, \ldots, R_K \in \textbf{SO}(3)} & f(R) = \sum_{i \neq j} \|R_i c_{ij} - R_j c_{ji}\|_2 \\
    & \text{subject to} & \widetilde{R}^T \widetilde{R} = \alpha K I_3, \quad 1 \leq i \leq K.
\end{aligned}
\end{equation}
We have adapted the Riemannian subgradient method by substituting the QR retraction with alternating projection onto two Riemannian manifolds: $R_1, \ldots, R_K \in \textbf{SO}(3)$ and $\widetilde{R}^T \widetilde{R} = \alpha K I_3$. This adaptation, called ReSync-norm, demonstrates its effectiveness in preventing clustering viewing directions, as visualized in \Cref{fig:vis-direction} (c).

\section{Convergence analysis} \label{sec:convergence}
In this section, we demonstrate the global convergence of our proposed Riemannian subgradient method, including its stochastic variants: stochastic Riemannian subgradient, block coordinate Riemannian subgradient, and block stochastic Riemannian subgradient methods. We establish that all these methods achieve an iteration complexity of $\mathcal{O}(\varepsilon^{-4})$ to reduce a chosen stationarity measure below $\varepsilon$. This complexity guarantee coincides with estimates provided in \cite{davis2019stochastic} and \cite{li2021weakly} for several algorithms designed to solve weakly convex minimization problems. To the best of our knowledge, this is the first complexity result for the block stochastic Riemannian subgradient method.

Considering the influence of Stiefel manifold constraints on the problem, we adopt a stationarity measure that corresponds to the gradient of the Moreau envelope for weakly convex functions, as detailed in \cite{li2021weakly}. This choice of stationarity measure draws inspiration from recent research on weakly convex minimization in Euclidean spaces, as discussed in\cite{davis2019stochastic,drusvyatskiy2019efficiency}. First, let us define analogs of the Moreau envelope as 
\begin{align}\label{eq:Moreau envelope}
    f_{\lambda}(X) = \min_{Y\in \text{Stiefel}(3K,3)} \left\{ f(Y) + \frac{1}{2\lambda} \|Y - X\|_F^2 \right\}, \quad X \in \text{Stiefel}(3K,3),
\end{align}
and the proximal mapping for any \( \lambda > 0 \) for the problem \eqref{formula: LUD}:
\begin{equation}
P_{\lambda f} (X) = \arg\min_{Y\in \text{Stiefel}(3K,3)} \left\{ f(Y) + \frac{1}{2\lambda} \|Y - X\|_F^2 \right\}, \quad X \in \text{Stiefel}(3K,3).
\end{equation}
We need to notice that, we call $X$ a \emph{stationary point} of problem \eqref{formula: LUD} if $X \in \text{Stiefel}(3K,3)$ satisfies the following first-order optimality condition: 
\begin{equation}\label{eq:first order optimality}
{0} \in \partial_{\mathbf{R}}f(X).
\end{equation}
According to the first-order optimality condition of $P_{\lambda f} (X)$, \eqref{ineq:proxi-opt} and \eqref{ineq:dist-theta} show that 
\begin{equation}
\label{eq:surrogate optimality}
    \text{dist} \left( 0,\partial_{\mathcal{R}} f \left(P_{\lambda f} (X)\right) \right) \leq \lambda^{-1} \cdot \left\|  P_{\lambda f} (X) - X \right\|_F =: \Theta(X).
\end{equation}
Specifically, from \eqref{eq:surrogate optimality}, we observe that when $\Theta(X) = 0$, we have $P_{\lambda f} (X) = X$ and thus $ \text{dist} \left( 0,\partial_{\mathcal{R}} f \left(X\right) \right) = 0$, which implies $X$ is a stationary point of problem \eqref{formula: LUD}. This motivates us to use $X \mapsto \Theta(X)$ as a stationarity measure for problem \eqref{formula: LUD}, and we call $X\in\text{Stiefel}(3K,3)$ an \emph{$\varepsilon$-approximate stationary point} of problem \eqref{formula: LUD} if it satisfies $\Theta(X)\leq \varepsilon$.

Since the Riemannian subgradient, the stochastic Riemannian subgradient, and the block coordinate Riemannian subgradient methods are all special cases of the block stochastic Riemannian subgradient method. So it is enough to estimate the iteration complexity of Algorithm \ref{alg:stochastic}.
\begin{theorem}\label{thm:convergence}
Assuming a constant step size, \(\mu_t = \frac{1}{\sqrt{T+1}}\), \((k=0,1,\ldots)\), where \(T\) represents the total number of iterations, Algorithm \ref{alg:stochastic} uniformly and randomly selects an index \(\overline{t}\) from \(\{1,\ldots,T\}\) and returns the corresponding \(X_{\overline{t}}\). Then, we have
$$
E\left[ \Theta^2\left(X_{\overline{t}}\right) \right] \leq  \frac{c_1\big(f_{\lambda}(X_0) - \min f_{\lambda}\big) + c_2}{\sqrt{T+1}},
$$
where the constants $c_1$ and $c_2$ are given in \eqref{eq:c1c2}. In particular, the iteration complexity of Algorithm \ref{alg:stochastic} for computing an $\varepsilon$-nearly stationary point of the problem \eqref{formula: LUD} is $O(\varepsilon^{-4})$.
\end{theorem}
The proof can be found in Appendix \ref{app:conv-analysis}.

\section{Experiments} \label{sec:exp}

To evaluate the performance of the ReSync algorithm and its stochastic variants, we conduct experiments on both simulated rotations and noisy projections of particles. For benchmark methods, we choose eigenvector relaxation \cite{singer2011three}, PGD method for LS and LUD formulations \cite{pan2023orientation}, as well as IRLS method \cite{wang2013orientation}. All experiments are executed on a machine with AMD Ryzen 9 7950X CPU with 16 cores, running at 4.50 GHz. In addition, all results are averaged over 10 independent trials. Our code is available at \url{https://github.com/zwyhahaha/ReSync4CryoEM.git}.

In line with prior studies, we use mean-squared error (MSE) as the metric to assess the accuracy of estimated orientations. The MSE, calculated between estimated rotations \(\hat{R}_1,..., \hat{R}_K\) and the underlying ground-truth rotations \(R_1,..., R_K\),  is defined as
\begin{align} \label{eq:MSE}
    \text{MSE} = \min_{O \in \text{SO}(3)} \frac{1}{K} \sum_{i=1}^K \|R_i - O \hat{R}_i\|_F^2.
\end{align} 
\vspace{-0.4cm}
\subsection{Synthetic data}

\begin{figure}
\vspace{-0.3cm}
    \centering
    \subfigure{\includegraphics[width=0.25\textwidth]{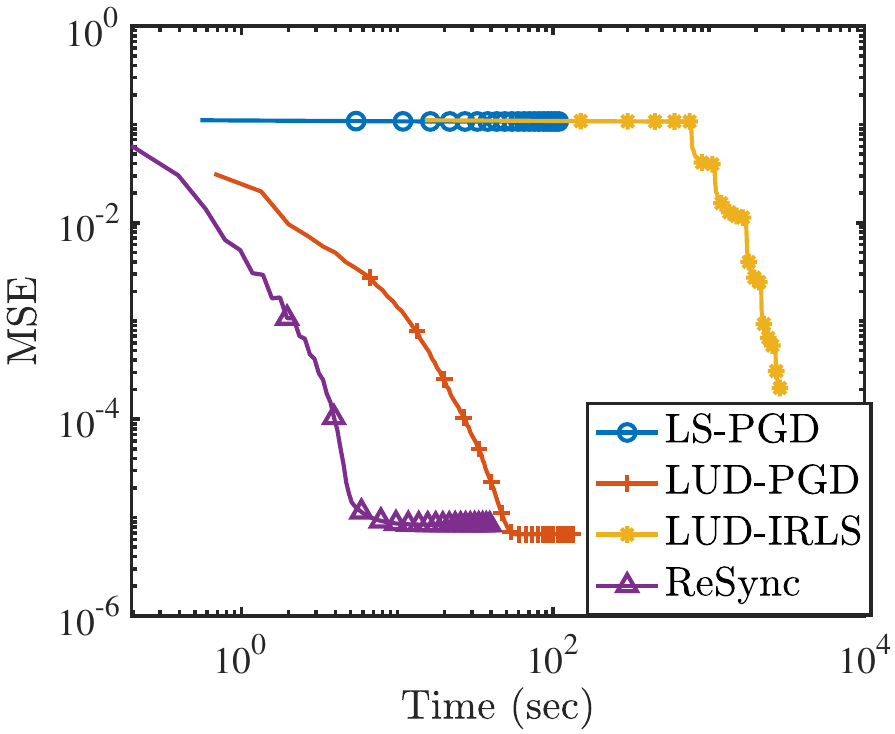}}
    \hspace{0.5 cm}
    \subfigure{\includegraphics[width=0.25\textwidth]{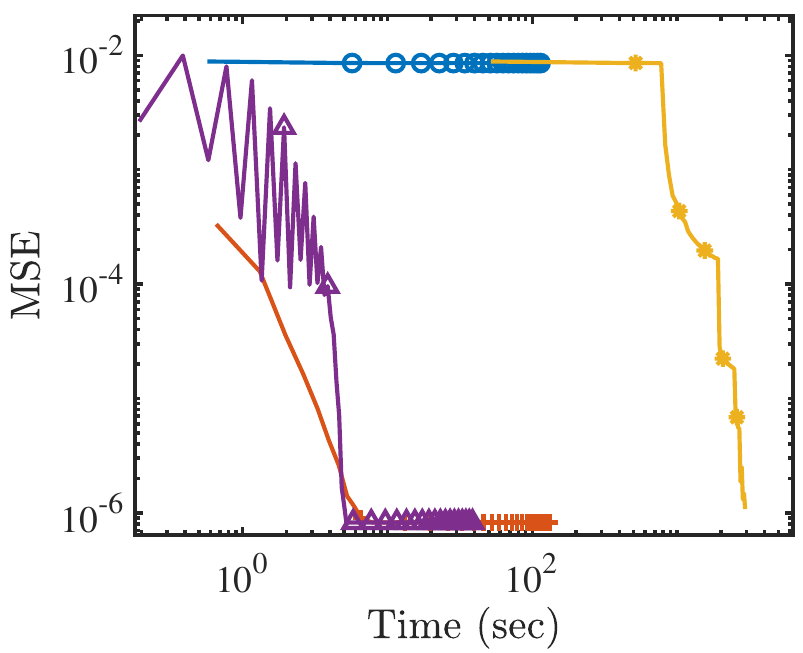}}
    \hspace{0.5 cm}
    \subfigure{\includegraphics[width=0.26\textwidth]{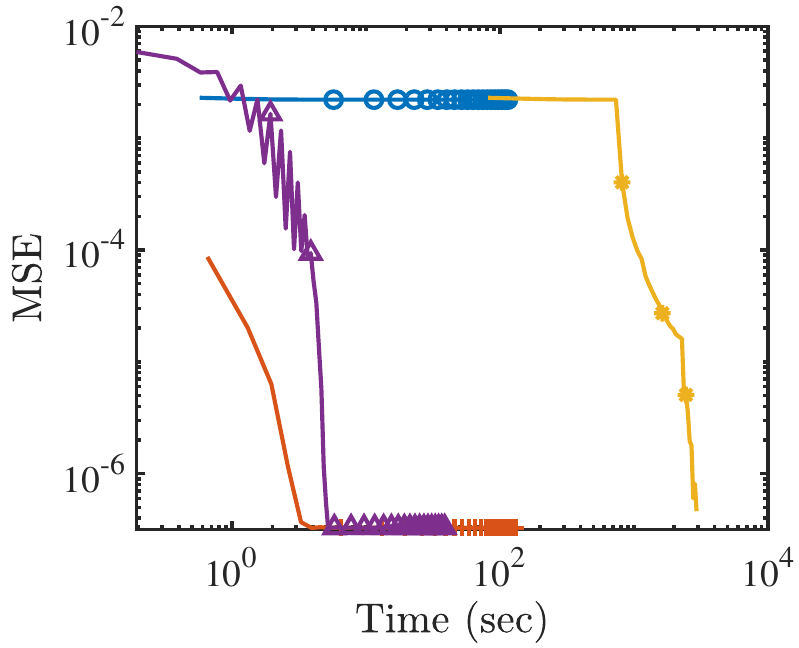}}
    \vspace{-0.2cm}
    \caption{Comparison with other orientation determination algorithms under synthetic data ($K=3000$). From left to right, the detection rate are $p=0.1, 0.3, 0.5$, respectively.}
    \label{fig:plot_benchmark}
    \vspace{-0.1cm}
\end{figure}

\begin{figure}
    \centering
    \subfigure{\includegraphics[width=0.25\textwidth]{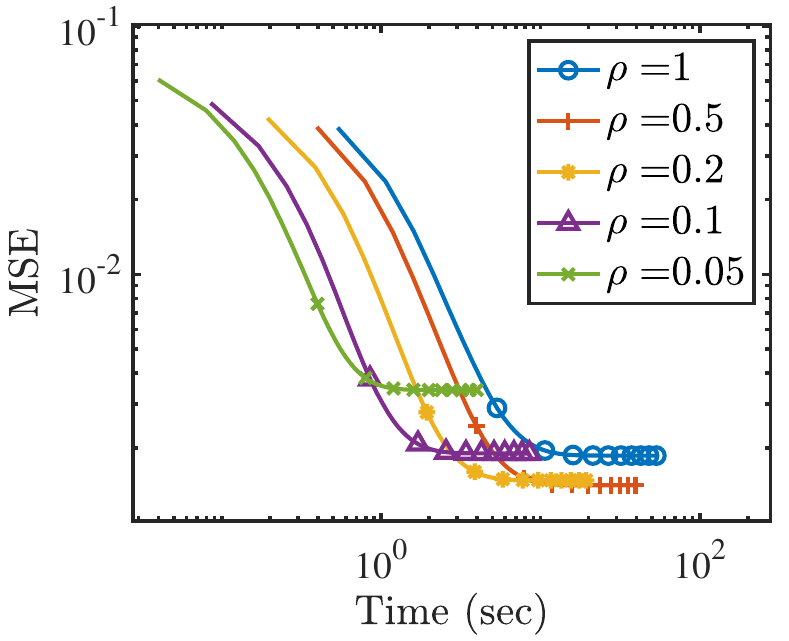}}
    \hspace{0.5 cm}
    \subfigure{\includegraphics[width=0.25\textwidth]{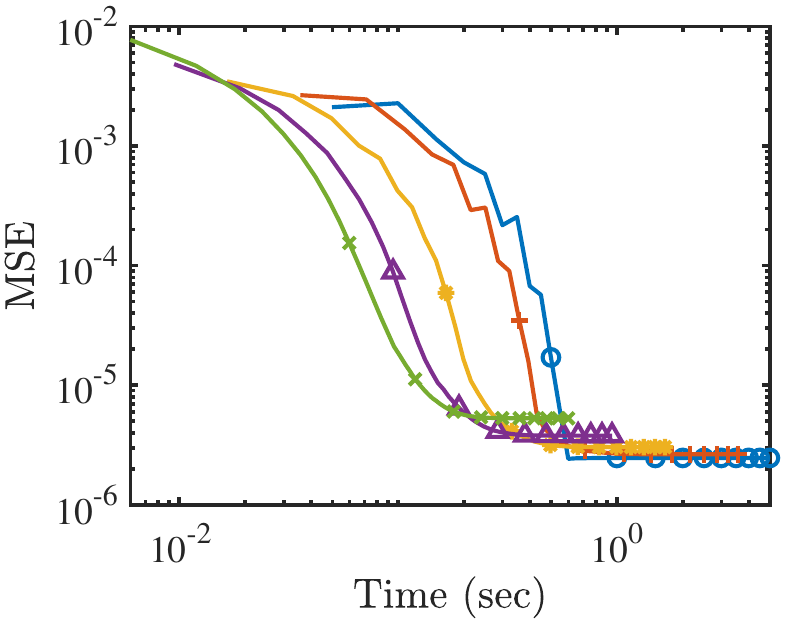}}
    \hspace{0.5 cm}
    \subfigure{\includegraphics[width=0.26\textwidth]{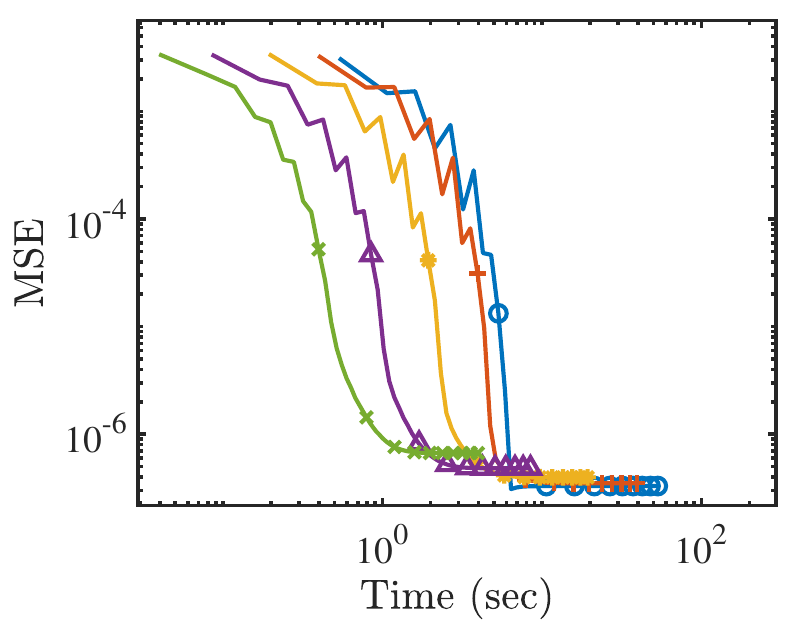}}
    \caption{Convergence of ReSync-BSGD under synthetic data ($K=3000$). From left to right, the detection rate are $p=0.1, 0.3, 0.5$, respectively.}
    \label{fig:plot_random}
\end{figure}

\textbf{Settings.} To assess the efficiency of different algorithms, we conduct experiments using simulated rotations. First, we generate \(K\) ground-truth rotations uniformly from \(\text{SO}(3)\) and calculate pairwise common-line indexes with angular resolution \(n_{\theta}=360\). To simulate the corrupted nature of Cryo-EM images, we introduce a detection rate \(p\), referring to the probability that common lines are correctly detected. Specifically, each common-line pair \((c_{ij},c_{ji})\) remains unchanged with probability \(p\), or is replaced by a uniform random pair with probability \(1-p\). We test values of $p$ at \((0.5,0.3,0.1,0.05)\) and $K$ at \((3000,5000)\).  Notably, simulations at \(p=0.05\) or \(K=5000\) are rarely seen in previous studies. 

\textbf{Convergence verification.} In \Cref{fig:plot_benchmark}, we illustrate the convergence behavior of ReSync algorithms alongside other benchmarks for \(K=3000\) under detection rates \(p=0.1, 0.3, 0.5\). As the figure shows, ReSync algorithm converges rapidly to a high-accuracy solution, particularly at lower 
$p$ values. In \Cref{fig:plot_random}, we display the convergence behavior of ReSync-BSGD with varying filter ratios $\rho$ under identical settings. The figure illustrates a consistent trade-off between efficiency and accuracy: as $\rho$ decreases, the convergence rate improves at the expense of accuracy. Additionally, as the detection rate increases, the accuracy differences among various $\rho$ values become less pronounced.

\textbf{Efficiency benchmarks.} The MSE and running times for various methods with \(K=3000\) and \(K = 5000\) are listed in \Cref{tab:simu}. For consistent comparison, we set the same stopping precision for all methods, achievable within the maximum iteration count for most methods. Every method is warm-started with the solution from eigenvector relaxation. As shown in \Cref{tab:simu}, ReSync-BSGD outperforms other methods in most scenarios by achieving comparable accuracy and a 5 to 30-fold speedup compared to state-of-the-art methods.

\begin{table}[]
\caption{MSE and running time (in seconds) for different methods under simulated data. Here $P$ refers to the detection rate, $K$ is the projection number, and the number in the bracket is the target precision. For stochastic ReSync algorithms, the filter ratio is 0.1.}
\vspace{0.2cm}
\label{tab:simu}
\centering
\small
\resizebox{0.9\textwidth}{!}{
\renewcommand{\arraystretch}{1.1}
\begin{tabular}{|c|c|c|c|c|c|c|c|c|}
\hline
\multicolumn{9}{|c|}{K=3000} \\
\hline
\multicolumn{1}{|c|}{P} & \multicolumn{2}{c|}{0.5/5e-7} & \multicolumn{2}{c|}{0.3/5e-6} & \multicolumn{2}{c|}{0.1/5e-3} & \multicolumn{2}{c|}{0.05/0.4} \\
\hline
Method & MSE & Time & MSE & Time & MSE & Time & MSE & Time \\
\hline
Eig & 2.67E-03 & 14.16 & 9.41E-03 & 14.12 & 1.11E-01 & 14.29 & 6.97E-01& 13.35
\\
LS-PGD & 2.20E-03 & 127.53 & 8.86E-03 & 128.00 & 1.07E-01 & 118.09 & 6.15E-01& 106.60
\\
LUD-PGD & 3.26E-07 & 9.57 & 1.57E-06 & 9.45 & 3.07E-03 & 9.27 & 3.30E-01& 7.81
\\
LUD-IRLS& 3.29E-07 & 25.27 & \textbf{8.34E-07} & 34.62 & \textbf{4.83E-05} & 124.50 & \textbf{1.57E-01}& 167.87
\\
ReSync & 3.25E-07 & 20.29 & 1.15E-06 & 19.71 & 1.78E-03 & 15.99 & 4.23E-01& 14.51
\\
ReSync-BCD& \textbf{3.20E-07} & 9.66 & 1.29E-06 & 8.12 & 2.42E-03 & 5.73 & 4.49E-01& 3.61
\\
ReSync-SGD& 4.12E-07 & 7.77 & 2.37E-06 & 4.12 & 3.25E-03 & 5.53 & 4.55E-01& 3.90
\\
ReSync-BSGD& 4.76E-07 & \textbf{6.47} & 1.35E-06 & \textbf{3.34} & 1.85E-03 & \textbf{3.20} & 4.35E-01& \textbf{2.43}\\
\hline
\multicolumn{9}{|c|}{K=5000} \\
\hline
\multicolumn{1}{|c|}{P} & \multicolumn{2}{c|}{0.5 (5e-7)} & \multicolumn{2}{c|}{0.3 (1e-6)} & \multicolumn{2}{c|}{0.1 (1e-4)} & \multicolumn{2}{c|}{0.05 (0.1)} \\
\hline
Method & MSE & Time & MSE & Time & MSE & Time & MSE & Time \\
\hline
Eig & 1.43E-03 & 39.19 & 5.60E-03& 37.39
& 6.45E-02& 38.79
& 3.31E-01 & 38.62 \\
LS-PGD & 1.31E-03 & 324.48 & 5.28E-03& 315.64
& 6.31E-02& 354.65
& 6.07E-01 & 299.28 \\
LUD-PGD & 1.93E-07 & 26.25 & 5.10E-07& 24.26
& 1.41E-04& 43.31
& \textbf{1.57E-02} & 22.64 \\
LUD-IRLS& 1.95E-07 & 57.15 & 4.96E-07& 87.58
&\textbf{ 1.26E-05}& 250.29
& 4.55E-02 & 573.19 \\
ReSync & 3.65E-07 & 53.91 & 7.48E-07& 51.56
& 3.21E-04& 47.13
& 3.98E-02 & 41.90 \\
ReSync-BCD& \textbf{1.90E-07} & 25.93 & \textbf{4.82E-07}& 24.83
& 3.42E-04& 15.28
& 9.12E-02 & 15.22 \\
ReSync-SGD& 3.02E-07 & 16.04 & 3.61E-06& \textbf{5.19}
& 3.52E-04& 12.43
& 5.34E-02 & 16.24 \\
ReSync-BSGD& 3.13E-07 & \textbf{8.69 }& 1.35E-06& 5.42& 8.48E-05& \textbf{8.40}& 9.41E-02 & \textbf{8.53} \\
\hline
\end{tabular}
}
\vspace{-0.4cm}
\end{table}

\begin{figure}
\vspace{-0.5cm}
    \centering
    \subfigure[clean]{\includegraphics[width=0.14\textwidth]{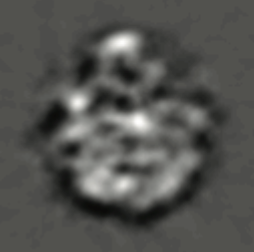}}
    \hspace{0.3 cm}
    \subfigure[SNR=1/16]{\includegraphics[width=0.14\textwidth]{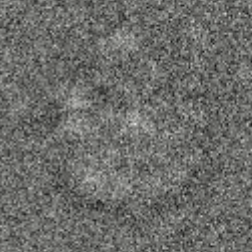}}
    \hspace{0.3 cm}
    \subfigure[SNR=1/32]{\includegraphics[width=0.14\textwidth]{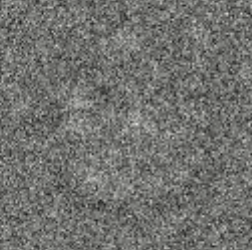}}
    \hspace{0.3 cm}
    \subfigure[SNR=1/64]{\includegraphics[width=0.14\textwidth]{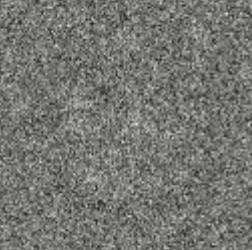}}
    \hspace{0.3 cm}
    \subfigure[SNR=1/128]{\includegraphics[width=0.14\textwidth]{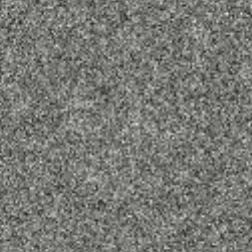}}
    \caption{Projection images of 50-S ribosomal subunit under different SNR levels.}
    \label{fig:SNR}
    \vspace{-0.5cm}
\end{figure}

\subsection{Real data}
\textbf{Setting.} To test the real-world applicability of different methods, we conduct experiments using real projection data. In line with prior studies \cite{wang2013orientation,pan2023orientation}, we select 50-S ribosomal subunit with volume \(65\times 65 \times 65\) as the observing target, whose 3D density map is available. We start by sampling ground-truth rotations \(R_1,..., R_K\), then \(K\) clean projection images with pixel size \(129\times 129\) are generated from the underlying density map. To simulate noisy Cryo-EM images, we add Gaussian noise to these images according to certain SNR. Subsequently, common-lines are estimated from these noisy projections by conducting Fourier transform (angular resolution \(n_{\theta}=360\), radial resolution \(n_r=100\)) and finding the most similar pair of Fourier coefficients. Considering that Cryo-EM images feature low SNR and large numbers, we set \(\text{SNR}=(1/16, 1/32, 1/64, 1/128)\) and \(K=(3000,5000)\). To the best of our knowledge, this marks the first valid documentation of experimental results for \(\text{SNR}=1/128\) in the literature. Additionally, we present projection images of the 50-S ribosomal subunit under this SNR setting in \Cref{fig:SNR}. Notably, at an \(\text{SNR}\) of 1/128, the images are predominantly obscured by significant noise.

\textbf{Efficiency benchmarks.} The experiment results for \(K=3000\) and \(K = 5000\) are shown in \Cref{tab:real}. The stopping criteria are triggered when the relative distance of adjacent iterates is smaller than \(10^{-5}\). All methods are warm-started by eigenvector relaxations. As these tables show, when SNR is \(1/16\) or \(1/32\), most methods converge to a high-accuracy solution, and ReSync-BSGD achieves 15-100 times acceleration compared with LUD-PGD and LUD-IRLS. When \(\text{SNR}=1/64\) and \(K=3000\), all benchmark methods diverge in precision but our stochastic ReSync algorithms can prevent divergence to some extent. However, if SNR decreases to \(1/128\), eigenvector relaxation fails to produce high-quality initialization points and all methods perform poorly in accuracy. 
\begin{table}[ht]
\caption{MSE and running time (in seconds) for different methods under real-world data. Here SNR refers to signal-to-noise ratio and $K$ is the projection number. For stochastic ReSync algorithms, the filter ratio is 0.1.}
\vspace{0.2cm}
\label{tab:real}
\centering
\small 
\renewcommand{\arraystretch}{1.1}
\resizebox{0.9\textwidth}{!}{
\begin{tabular}{|c|c|c|c|c|c|c|c|c|}
\hline
\multicolumn{9}{|c|}{K=3000} \\
\hline
\multicolumn{1}{|c|}{SNR} & \multicolumn{2}{c|}{1/16} & \multicolumn{2}{c|}{1/32} & \multicolumn{2}{c|}{1/64} & \multicolumn{2}{c|}{1/128} \\
\hline
Method & MSE & Time & MSE & Time & MSE & Time & MSE & Time \\
\hline
Eig & 3.28E-02 & 17.40 & 1.02E-01 & 17.12 & 5.96E-01 & 17.02 & 2.71E+00 & 17.10 \\
LS-PGD & 1.00E-01 & 32.60 & 4.60E-01 & 143.96 & 1.21E+00 & 201.93 & 2.25E+00 & 200.73 \\
LUD-PGD & \textbf{6.81E-03} & 45.60 & 6.26E-02 & 75.98 & 2.05E+00 & 203.49 & 2.17E+00 & 212.24 \\
LUD-IRLS& 7.13E-03 & 153.60 & \textbf{4.92E-02} & 471.80 & 2.05E+00 & 820.58 & 2.20E+00 & 772.56 \\
ReSync & 6.89E-03 & 23.70 & 5.99E-02 & 25.84 & 9.58E-01 & 27.58 & 2.18E+00 & 30.16 \\
ReSync-BCD& 7.26E-03 & 6.70 & 5.15E-02 & 8.64 & 6.15E-01 & 8.16 & \textbf{2.16E+00} & 11.02 \\
ReSync-SGD& 7.32E-03 & 7.60 & 5.18E-02 & 9.74 & \textbf{5.79E-01} & 10.00 & \textbf{2.16E+00} & 13.52 \\
ReSync-BSGD& 7.26E-03 & \textbf{2.90} & 5.51E-02 & \textbf{3.05} & 5.88E-01 & \textbf{2.72} & \textbf{2.16E+00} & \textbf{3.02} \\
\hline
\multicolumn{9}{|c|}{K=5000} \\
\hline
\multicolumn{1}{|c|}{SNR} & \multicolumn{2}{c|}{1/16} & \multicolumn{2}{c|}{1/32} & \multicolumn{2}{c|}{1/64} & \multicolumn{2}{c|}{1/128} \\
\hline
Method & MSE & Time & MSE & Time & MSE & Time & MSE & Time \\
\hline
Eig & 3.08E-02 & 47.70 & 4.19E+00 & 47.70 & 4.86E+00 & 48.69 & 3.11E+00 & 48.68 \\
LS-PGD & 1.06E-01 & 125.00 & 3.52E+00 & 580.34 & 2.50E+00 & 760.39 & 2.39E+00 & 794.19 \\
LUD-PGD &\textbf{ 6.70E-03} & 110.23 & 5.22E-02 & 651.85 & \textbf{2.31E+00} & 849.83 & \textbf{2.25E+00} & 848.16 \\
LUD-IRLS& 7.00E-03 & 432.60 & 7.16E-02 & 1388.64 & 2.33E+00 & 2915.65 & 2.29E+00 & 2929.07 \\
ReSync & 6.83E-03 & 61.90 & 5.51E-02 & 93.36 & 3.46E+00 & 77.48 & \textbf{2.25E+00} & 82.98 \\
ReSync-BCD& 7.14E-03 & 14.00 & 5.48E-02 & 16.00 & 3.90E+00 & 23.84 & 2.29E+00 & 34.54 \\
ReSync-SGD& 6.97E-03 & 13.40 & \textbf{5.15E-02} & 16.74 & 3.95E+00 & 24.83 & 2.29E+00 & 35.01 \\
ReSync-BSGD& 7.01E-03 & \textbf{5.70} & 5.36E-02 & \textbf{6.27} & 3.92E+00 & \textbf{9.61} & 2.29E+00 & \textbf{7.79} \\
\hline
\end{tabular}
}
\end{table}

\subsubsection{Techniques in tackling noise}

Previous results reveal that when noise levels are high (with $p=0.05$ or $\text{SNR}=1/64, 1/128$), most methods get stuck at solutions of low accuracy. As discussed in \Cref{sec:norm}, this could be addressed by applying a spectral norm constraint and image denoising techniques to enhance performance. 

To assess the efficacy of norm-constrained methods (i.e., SDR-ADMM and ReSync-norm), we conducted experiments on real data, setting the projection number at $K=3000$ across various SNR levels. We chose a hyperparameter \(\alpha = \frac{2}{3}\), based on the prior knowledge that viewing directions are uniformly distributed in large samples. The results, detailed in \Cref{tab:norm}, indicate that both SDR-ADMM and ReSync-norm perform well at lower SNRs, with SNR\(=1/64\), yielding relatively low MSEs. Notably, ReSync-norm maintains accurate and stable performance as SNR ranges from $1/16$ to $1/64$, whereas SDR-ADMM struggles with higher SNRs. Additionally, according to \Cref{tab:norm-time}, ReSync-norm significantly outperforms SDR-ADMM in computational efficiency, achieving a 25 to 60-fold speed increase.

While norm-constrained methods like SDR-ADMM and ReSync-norm demonstrate superior performance when SNR\(\geq 1/64\), their effectiveness diminishes when faced with extremely high noise levels, such as SNR\(= 1/128\). To mitigate the reduced effectiveness of norm-constrained methods at extremely high noise levels, we have explored the application of image-denoising techniques. Class averaging, the most prevalent preprocessing technique for Cryo-EM images as outlined by \cite{frank2006three}, clusters images with similar viewing angles and averages them to enhance signal clarity. For this clustering, we utilize Vector Diffusion Maps (VDM) \cite{singer2012vector}, which effectively organize complex data in low-dimensional spaces. The results of these experiments, displayed in the last column of \Cref{tab:norm}, indicate that most methods achieve higher precision, likely due to the improved quality of initial data.

\begin{table}[ht]
\centering
\small
\caption{MSE for different methods under various SNR levels and denoised images. Here $K=3000$, $\alpha=2/3$, and methods with spectral norm constraint are in bold font.}
\vspace{0.2cm}
\label{tab:norm}
\renewcommand{\arraystretch}{1.0}
\setlength{\tabcolsep}{8pt} 
\resizebox{0.8\textwidth}{!}{
\begin{tabular}{ccccccc}
\toprule
SNR & 1/16 & 1/32 & 1/64 & 1/128 & 1/128, denoised \\
\midrule
Eig & 0.0328 & 0.1016 & 0.5962 & 2.7131 & 0.8241 \\
LS-PGD & 0.1001 & 0.4595 & 1.2104 & 2.2516 & 0.7518 \\
LUD-PGD & 0.0068 & 0.0626 & 2.0459 & \textbf{2.1717} & 4.9269 \\
LUD-IRLS-PGD & 0.0071 & 0.0904 & 2.0467 & 2.2004 & 4.3826 \\
\textbf{SDR-ADMM} & 0.0315 & 0.1283 & 0.3634 & 2.1946 & \textbf{0.7249} \\
ReSync & 0.0069 & 0.0599 & 0.9577 & 2.1768 & 0.7289 \\
\textbf{ReSync-norm} & \textbf{0.0062} & \textbf{0.0180} & \textbf{0.2495} & 2.3273 & 0.7282 \\
\bottomrule
\end{tabular}
}
\end{table}

\begin{table}[ht]
\centering
\small
\caption{Running time (in seconds) for different methods under various SNR levels. Here $K=3000$, $\alpha=2/3$, and methods with spectral norm constraint are in bold font.}
\vspace{0.2cm}
\label{tab:norm-time}
\renewcommand{\arraystretch}{1.0}
\setlength{\tabcolsep}{8pt} 
\resizebox{0.8\textwidth}{!}{
\begin{tabular}{cccccc}
\toprule
SNR & 1/16 & 1/32 & 1/64 & 1/128, denoised  \\
\midrule
Eig & 17.40 & 17.12 & 17.02 & 17.10  \\
LS-PGD & 32.59 & 143.96 & 201.93 & 200.73  \\
LUD-PGD & 45.60 & 75.98 & 203.49 & 212.24  \\
LUD-IRLS-PGD & 153.56 & 471.80 & 820.58 & 772.56 \\
\textbf{SDR-ADMM} & 13957.06 & 14883.00 & 13722.88 & 9377.16 \\
ReSync & 23.70 & 25.84 & 27.58 & 30.16 \\
\textbf{ReSync-norm} & 232.77 & 296.59 & 359.38 & 406.14\\
\bottomrule
\end{tabular}
}
\end{table}

\section{Conclusion and Limitations} \label{sec:conclusion}

In this paper, we tackle the orientation determination problem in Cryo-EM images. We adopt a common-line-based LUD formulation and propose to use the Riemannian subgradient method to solve this nonconvex and nonsmooth synchronization problem. For tackling a large number of images, we propose a block stochastic Riemannian subgradient method, which achieves a surprising acceleration effect and surpasses its SGD and BCD counterparts. To handle the noisy images, we incorporate spectral norm constraint into our algorithm and apply an alternating projection method to guarantee feasibility. Experiments on both synthetic and real data underscore the superiority in effectiveness and efficiency of our proposed method.

We will discuss the limitations from both empirical and theoretical aspects. In our experiments on real data, we have established the effectiveness of the ReSync-norm algorithm. However, the alternative projection in each iteration is costly in computation. The remaining issue is to develop a more efficient algorithm for tackling norm constraints. Theoretically, we have established the convergence results of ReSync-BSGD, but an in-depth analysis of the convergence rate, especially its superiority over SGD and BCD, is not involved in our result. As a consequence, another research direction is to explore the property of the BSGD algorithm, especially out of the scope of some specific synchronization problem.

\bibliographystyle{plainnat}
\bibliography{ref} 

\begin{thebibliography}{33}
\providecommand{\natexlab}[1]{#1}
\providecommand{\url}[1]{\texttt{#1}}
\expandafter\ifx\csname urlstyle\endcsname\relax
  \providecommand{\doi}[1]{doi: #1}\else
  \providecommand{\doi}{doi: \begingroup \urlstyle{rm}\Url}\fi

\bibitem[Absil et~al.(2008)Absil, Mahony, and Sepulchre]{absil2008optimization}
P-A Absil, Robert Mahony, and Rodolphe Sepulchre.
\newblock \emph{Optimization algorithms on matrix manifolds}.
\newblock Princeton University Press, 2008.

\bibitem[Amari(1993)]{amari1993backpropagation}
Shun-ichi Amari.
\newblock Backpropagation and stochastic gradient descent method.
\newblock \emph{Neurocomputing}, 5\penalty0 (4-5):\penalty0 185--196, 1993.

\bibitem[Bandeira et~al.(2020)Bandeira, Chen, Lederman, and Singer]{bandeira2020non}
Afonso~S Bandeira, Yutong Chen, Roy~R Lederman, and Amit Singer.
\newblock Non-unique games over compact groups and orientation estimation in cryo-em.
\newblock \emph{Inverse Problems}, 36\penalty0 (6):\penalty0 064002, 2020.

\bibitem[Bendory et~al.(2020)Bendory, Bartesaghi, and Singer]{bendory2020single}
Tamir Bendory, Alberto Bartesaghi, and Amit Singer.
\newblock Single-particle cryo-electron microscopy: Mathematical theory, computational challenges, and opportunities.
\newblock \emph{IEEE signal processing magazine}, 37\penalty0 (2):\penalty0 58--76, 2020.

\bibitem[Chen et~al.(2020)Chen, Ma, Man-Cho~So, and Zhang]{chen2020proximal}
Shixiang Chen, Shiqian Ma, Anthony Man-Cho~So, and Tong Zhang.
\newblock Proximal gradient method for nonsmooth optimization over the stiefel manifold.
\newblock \emph{SIAM Journal on Optimization}, 30\penalty0 (1):\penalty0 210--239, 2020.

\bibitem[Crowther et~al.(1970)Crowther, Amos, Finch, De~Rosier, and Klug]{crowther1970three}
RAea Crowther, Linda~A Amos, JT~Finch, DJ~De~Rosier, and A~Klug.
\newblock Three dimensional reconstructions of spherical viruses by fourier synthesis from electron micrographs.
\newblock \emph{Nature}, 226\penalty0 (5244):\penalty0 421--425, 1970.

\bibitem[Davis and Drusvyatskiy(2019)]{davis2019stochastic}
Damek Davis and Dmitriy Drusvyatskiy.
\newblock Stochastic model-based minimization of weakly convex functions.
\newblock \emph{SIAM Journal on Optimization}, 29\penalty0 (1):\penalty0 207--239, 2019.

\bibitem[Drusvyatskiy and Paquette(2019)]{drusvyatskiy2019efficiency}
Dmitriy Drusvyatskiy and Courtney Paquette.
\newblock Efficiency of minimizing compositions of convex functions and smooth maps.
\newblock \emph{Mathematical Programming}, 178:\penalty0 503--558, 2019.

\bibitem[Dutt and Rokhlin(1993)]{dutt1993fast}
Alok Dutt and Vladimir Rokhlin.
\newblock Fast fourier transforms for nonequispaced data.
\newblock \emph{SIAM Journal on Scientific computing}, 14\penalty0 (6):\penalty0 1368--1393, 1993.

\bibitem[Fessler and Sutton(2003)]{fessler2003nonuniform}
Jeffrey~A Fessler and Bradley~P Sutton.
\newblock Nonuniform fast fourier transforms using min-max interpolation.
\newblock \emph{IEEE transactions on signal processing}, 51\penalty0 (2):\penalty0 560--574, 2003.

\bibitem[Frank(2001)]{frank2001cryo}
Joachim Frank.
\newblock Cryo-electron microscopy as an investigative tool: the ribosome as an example.
\newblock \emph{Bioessays}, 23\penalty0 (8):\penalty0 725--732, 2001.

\bibitem[Frank(2006)]{frank2006three}
Joachim Frank.
\newblock \emph{Three-dimensional electron microscopy of macromolecular assemblies: visualization of biological molecules in their native state}.
\newblock Oxford university press, 2006.

\bibitem[Gustafsson(1996)]{gustafsson1996mathematics}
Bertil Gustafsson.
\newblock Mathematics for computer tomography.
\newblock \emph{Physica Scripta}, 1996\penalty0 (T61):\penalty0 38, 1996.

\bibitem[Huang and Wei(2022)]{huang2022riemannian}
Wen Huang and Ke~Wei.
\newblock Riemannian proximal gradient methods.
\newblock \emph{Mathematical Programming}, 194\penalty0 (1):\penalty0 371--413, 2022.

\bibitem[Li et~al.(2021)Li, Chen, Deng, Qu, Zhu, and Man-Cho~So]{li2021weakly}
Xiao Li, Shixiang Chen, Zengde Deng, Qing Qu, Zhihui Zhu, and Anthony Man-Cho~So.
\newblock Weakly convex optimization over stiefel manifold using riemannian subgradient-type methods.
\newblock \emph{SIAM Journal on Optimization}, 31\penalty0 (3):\penalty0 1605--1634, 2021.

\bibitem[Li et~al.(2023)Li, Milzarek, and Qiu]{li2023convergence}
Xiao Li, Andre Milzarek, and Junwen Qiu.
\newblock Convergence of random reshuffling under the kurdyka--{\l}ojasiewicz inequality.
\newblock \emph{SIAM Journal on Optimization}, 33\penalty0 (2):\penalty0 1092--1120, 2023.

\bibitem[Liu et~al.(2023)Liu, Li, and So]{liu2023resync}
Huikang Liu, Xiao Li, and Anthony Man-Cho So.
\newblock Resync: Riemannian subgradient-based robust rotation synchronization.
\newblock \emph{arXiv preprint arXiv:2305.15136}, 2023.

\bibitem[Lyumkis(2019)]{lyumkis2019challenges}
Dmitry Lyumkis.
\newblock Challenges and opportunities in cryo-em single-particle analysis.
\newblock \emph{Journal of Biological Chemistry}, 294\penalty0 (13):\penalty0 5181--5197, 2019.

\bibitem[Pan et~al.(2023)Pan, Lu, Wen, Xu, and Zeng]{pan2023orientation}
Huan Pan, Jian Lu, You-Wei Wen, Chen Xu, and Tieyong Zeng.
\newblock Orientation estimation of cryo-em images using projected gradient descent method.
\newblock \emph{Inverse Problems}, 39\penalty0 (4):\penalty0 045002, 2023.

\bibitem[Radermacher et~al.(1986)Radermacher, Wagenknecht, Verschoor, and Frank]{radermacher1986new}
M~Radermacher, T~Wagenknecht, A~Verschoor, and J~Frank.
\newblock A new 3-d reconstruction scheme applied to the 50s ribosomal subunit of e. coli.
\newblock \emph{Journal of microscopy}, 141\penalty0 (1):\penalty0 RP1--RP2, 1986.

\bibitem[Rosenthal(2015)]{rosenthal2015high}
Peter~B Rosenthal.
\newblock From high symmetry to high resolution in biological electron microscopy: a commentary on crowther (1971)‘procedures for three-dimensional reconstruction of spherical viruses by fourier synthesis from electron micrographs’.
\newblock \emph{Philosophical Transactions of the Royal Society B: Biological Sciences}, 370\penalty0 (1666):\penalty0 20140345, 2015.

\bibitem[Safran and Shamir(2020)]{safran2020good}
Itay Safran and Ohad Shamir.
\newblock How good is sgd with random shuffling?
\newblock In \emph{Conference on Learning Theory}, pages 3250--3284. PMLR, 2020.

\bibitem[Shkolnisky and Singer(2012)]{shkolnisky2012viewing}
Yoel Shkolnisky and Amit Singer.
\newblock Viewing direction estimation in cryo-em using synchronization.
\newblock \emph{SIAM journal on imaging sciences}, 5\penalty0 (3):\penalty0 1088--1110, 2012.

\bibitem[Singer and Shkolnisky(2011)]{singer2011three}
Amit Singer and Yoel Shkolnisky.
\newblock Three-dimensional structure determination from common lines in cryo-em by eigenvectors and semidefinite programming.
\newblock \emph{SIAM journal on imaging sciences}, 4\penalty0 (2):\penalty0 543--572, 2011.

\bibitem[Singer and Wu(2012)]{singer2012vector}
Amit Singer and H-T Wu.
\newblock Vector diffusion maps and the connection laplacian.
\newblock \emph{Communications on pure and applied mathematics}, 65\penalty0 (8):\penalty0 1067--1144, 2012.

\bibitem[Singer et~al.(2010)Singer, Coifman, Sigworth, Chester, and Shkolnisky]{singer2010detecting}
Amit Singer, Ronald~R Coifman, Fred~J Sigworth, David~W Chester, and Yoel Shkolnisky.
\newblock Detecting consistent common lines in cryo-em by voting.
\newblock \emph{Journal of structural biology}, 169\penalty0 (3):\penalty0 312--322, 2010.

\bibitem[Tseng(2001)]{tseng2001convergence}
Paul Tseng.
\newblock Convergence of a block coordinate descent method for nondifferentiable minimization.
\newblock \emph{Journal of optimization theory and applications}, 109:\penalty0 475--494, 2001.

\bibitem[Van~Heel(1987)]{van1987angular}
Marin Van~Heel.
\newblock Angular reconstitution: a posteriori assignment of projection directions for 3d reconstruction.
\newblock \emph{Ultramicroscopy}, 21\penalty0 (2):\penalty0 111--123, 1987.

\bibitem[Wang et~al.(2013)Wang, Singer, and Wen]{wang2013orientation}
Lanhui Wang, Amit Singer, and Zaiwen Wen.
\newblock Orientation determination of cryo-em images using least unsquared deviations.
\newblock \emph{SIAM journal on imaging sciences}, 6\penalty0 (4):\penalty0 2450--2483, 2013.

\bibitem[Weissenberger et~al.(2021)Weissenberger, Henderikx, and Peters]{weissenberger2021understanding}
Giulia Weissenberger, Rene~JM Henderikx, and Peter~J Peters.
\newblock Understanding the invisible hands of sample preparation for cryo-em.
\newblock \emph{Nature Methods}, 18\penalty0 (5):\penalty0 463--471, 2021.

\bibitem[Xu and Yin(2015)]{xu2015block}
Yangyang Xu and Wotao Yin.
\newblock Block stochastic gradient iteration for convex and nonconvex optimization.
\newblock \emph{SIAM Journal on Optimization}, 25\penalty0 (3):\penalty0 1686--1716, 2015.

\bibitem[Yang et~al.(2014)Yang, Zhang, and Song]{yang2014optimality}
Wei~Hong Yang, Lei-Hong Zhang, and Ruyi Song.
\newblock Optimality conditions for the nonlinear programming problems on riemannian manifolds.
\newblock \emph{Pacific Journal of Optimization}, 10\penalty0 (2):\penalty0 415--434, 2014.

\bibitem[Zhao et~al.(2014)Zhao, Yu, Wang, Arora, and Liu]{zhao2014accelerated}
Tuo Zhao, Mo~Yu, Yiming Wang, Raman Arora, and Han Liu.
\newblock Accelerated mini-batch randomized block coordinate descent method.
\newblock \emph{Advances in neural information processing systems}, 27, 2014.

\end{thebibliography}

\appendix

\section{Proof of Theorem \ref{thm:convergence}} \label{app:conv-analysis}
We will follow the proof methodology outlined in \cite{li2021weakly}, starting with the introduction of the \emph{weak convexity inequality} under Riemannian manifold constraints as an essential tool. It is worth mentioning that the introduction of the \emph{weak convexity inequality} under Riemannian manifold constraints, which is used for the restrictions on weakly convex functions on the Stiefel manifold, is a fundamental component in the proof of convergence results in it. 

\begin{theorem}[\cite{li2021weakly}] \label{thm:Riemannian subgradient inequality}
Suppose \( h: \mathbb{R}^{3K \times 3} \rightarrow \mathbb{R} \) is \( \tau \)-weakly convex for some \( \tau \geq 0 \). Then, for any bounded open convex set \( \mathcal{U} \) containing  \( \mathcal{S}tiefel(3K, 3) \), there exists a constant \( L > 0 \) such that \( h \) is \( L \)-Lipschitz continuous on \( \mathcal{U} \), and satisfies the following conditions.
\begin{align}
		h(Y) &\geq h(X) + \left\langle  \widetilde \nabla_{\mathcal{R}} h(X)  , Y-X   \right\rangle  - \frac{\tau+L}{2} \|Y - X\|_F^2, \label{eq:Riemannian subgradient inequality}\\
		&\qquad \forall \ \widetilde \nabla_{\mathcal{R}} h(X) \in \partial_{\mathcal{R}} h(X) \ and \ X, Y \in \text{Stiefel}(3K,3)\notag
\end{align}
\end{theorem}

We can verify that the point $P_{\lambda f} (X)$ satisfies the first-order optimality condition:
\begin{equation}\label{ineq:proxi-opt}
0 \in \partial_{\mathcal{R}} f \left(P_{\lambda f} (X)\right) + \frac{1}{\lambda} \mathcal{P}_{\mathcal{T}_{P_{\lambda f} (X)}\text{St}} \left( P_{\lambda f} (X) - X\right).
\end{equation}
Therefore, we have
\begin{equation}\label{ineq:dist-theta}
\begin{split}
    \text{dist} \left( 0,\partial_{\mathcal{R}} f \left(P_{\lambda f} (X)\right) \right) &\leq \lambda^{-1}\cdot \left\| \mathcal{P}_{\mathcal{T}_{P_{\lambda f} (X)}\text{St}} \left( P_{\lambda f} (X) - X\right) \right\|_F \\
    &\leq \lambda^{-1} \cdot \left\|  P_{\lambda f} (X) - X \right\|_F =: \Theta(X).
\end{split}
\end{equation}

\subsection{Construction of the Riemannian Block Stochastic Subgradient Oracle}
With the tools ready, now we can start the algorithm's convergence analysis. Considering the block stochastic gradient descent (BSGD), the randomness comes from the random rearrangement of \(\{1,2,\cdots,K\}\) before each iteration, which leads to the random blocking of \(\{[1], \cdots, [\rho K]\}\). This affects the computation of the approximate subgradient and the update of variables for each block. 

It's important to note that we can choose \(\mathcal{S} = \mathcal{D}\), and considering the properties of the subgradient computation, we can immediately conclude that the computation of the approximate subgradient and the update of variables within each block are independent and self-contained. Let the sample space formed by the random partitioning of \(\{1,2,\cdots,K\}\) be denoted as\ \(\Omega \in \mathbf{R}^K\). Then, \(\zeta_t\) represents the outcome of the \(t\)-th random partitioning, which is distributed uniformly over \(\Omega\). We can  denoted it as \(\zeta_t \sim U(\Omega)\). 

Now, assume that the Riemannian random reshuffling block subgradient method is equipped with a \emph{Riemannian Block Stochastic Subgradient Oracle}, which possesses the following properties:\\
(a) The oracle can generate independent and identically distributed samples according to the distribution \(U(\Omega)\).\\
(b) Given a point \(X \in St(3K, 3)\), the oracle generates a sample \(\zeta \sim U(\Omega)\) and returns a block random approximate Riemannian subgradient
\begin{equation}
\label{eq:block random approximate Riemannian subgradient}
\widetilde{g}_{_\mathcal{R}}(X, \zeta) = 
\begin{pmatrix}
\widetilde \nabla_{R} g_1(X_1, \zeta) \\
\widetilde \nabla_{R} g_2(X_2, \zeta) \\
\vdots\\
\widetilde \nabla_{R} g_K(X_K, \zeta)
\end{pmatrix}
\end{equation}
\begin{remark}
In this setup, the random reshuffling of \(\{1, 2, \cdots, K\}\) results in blocks \([1], [2], \cdots, [m]\) with\ $m=\rho K$ and each block having a size of \(1/\rho\). For any \( i \in [n] \), \\
\begin{equation}
g_i(X_i,\zeta)=\sum_{\substack{j \in [n] \\ j \neq i}} f_{i, j}(X_i)   
\end{equation}
\end{remark}
Next, we will establish the relationship between the alternative stationary measure \(\Theta\) and the Moreau envelope \(f_\lambda\) in terms of sufficient reduction under expectation.
\begin{lemma}\label{lem:nonexpansive polar retra}    
Given \(X \in \text{Stiefel}(3K,3)\) and \(\xi \in T_X St\), consider a retraction based on QR decomposition, then it satisfies
\begin{align*}
\left\|\text{Retr}_{X}(\xi) - \overline X \right\|_F \le \|X + \xi  - \overline X\|_F + b \|\xi\|_F^2,
\end{align*}
\end{lemma}
\begin{proof}
The retraction on the compact submanifold \(\text{Stiefel}(3K,3)\) satisfies the \emph{second-order boundedness} property. So there exists a constant \(b \ge 0\) such that for all \(X \in \text{Stiefel}(3K,3)\) and \(\xi \in T_X St\),
\[ \| \text{Retr}_{X}(\xi) - X - \xi \|_F \le b \|\xi\|_F^2. \]
Then,
\begin{align*}
\left\|\text{Retr}_{X}(\xi) - \overline X \right\|_F 	&= \| (X + \xi)  - \overline X + \text{Retr}_{X}(\xi) - (X+\xi) \|_F \\
&\le \|X + \xi  - \overline X\|_F + b \|\xi\|_F^2,
\end{align*}
\end{proof}
\begin{proposition}
\label{eq:proposition6.3}    
Let \({X_t}\) be a sequence generated by the Riemannian block stochastic subgradient method (BSGD), with arbitrary initialization. Then, for any \(\lambda < \frac{c}{L+\tau}\), we have
\begin{equation}
\mu_tE\left[\Theta^2(X_t)\right] \le \frac{E\left[f_{\lambda}(X_t)\right] - E\left[f_{\lambda}(X_{t+1})\right] + \frac{KM^2}{\lambda}\mu_t^4+\frac{KB^2}{\lambda}\mu_t^2}{c\lambda\left(\frac{1}{\lambda} - \frac{L+\tau}{c}\right)},
\end{equation}
\end{proposition}
\begin{proof}
Applying the same block partitioning to $P_{\lambda f} (X^{t})$, we get $P_{\lambda f} (X^{t})_i \in \mathbb{R}^{3 \times 3}$, for each $i=1,2,\cdots,K$. Using~\eqref{eq:Moreau envelope}, the optimality of $P_{\lambda f} (X^{t+1})$,  \Cref{lem:nonexpansive polar retra}, and the fact that $P_{\lambda f} (X^{t}) \in \text{Stiefel}(3K,3)$, we obtain
\begin{align*}
   E_{\zeta_t\sim U(\Omega)
   }\left[f_\lambda(X^{t+1})\right]&\le f\left(P_{\lambda f}(X^t)\right)+\frac{1}{2\lambda}E_{\zeta_t\sim U(\Omega)}\left[\|P_{\lambda f}(X^t)-X^{t+1}\|_F^2\right]\\   
   &\le f\left(P_{\lambda f}(X^t)\right)+\frac{1}{2\lambda}\sum_{i=1}^{K}E_{\zeta_t\sim U(\Omega)}\left[\|P_{\lambda f}(X^t)_i-X^{t+1}_i\|_F^2\right]\\
   &\le f\left(P_{\lambda f}(X^t)\right)+\frac{1}{2\lambda}\sum_{i=1}^{K}E_{\zeta_t\sim U(\Omega)}\left[\left(\|X^t_i-\mu_t\widetilde \nabla_{R} g_i(X^t_i, \zeta_t)\|_F+M\mu_t^2\right)^2\right]\\
   &\le f\left(P_{\lambda f}(X^t)\right)+\frac{1}{\lambda}\sum_{i=1}^{K}E_{\zeta_t\sim U(\Omega)}\left[\|X^t_i-\mu_t\widetilde \nabla_{R} g_i(X^t_i, \zeta_t)-P_{\lambda f}(X^t)_i\|_F^2+M^2\mu_t^4\right]\\
   &\le f\left(P_{\lambda f}(X^t)\right)+\frac{2\mu_t}{\lambda}\sum_{i=1}^{K}E_{\zeta_t\sim U(\Omega)}\left[\left\langle  \widetilde\nabla_{R} g_{i}\left(X^t_i, \zeta_t\right) ,   P_{\lambda f} \left(X^{t}\right)_i - X^t_i\right\rangle\right]\\
   & \quad +\frac{KM^2}{\lambda}\mu_t^4+\frac{KB^2}{\lambda}\mu_t^2\\
 \end{align*}
The first inequality arises from the optimality of \( P_{\lambda f}(X_{t+1}) \). The third equation uses \Cref{lem:nonexpansive polar retra} as well as \(\left\|\widetilde\nabla_{R} g_i(X_i, \zeta)\right\| \leq \left\|\widetilde\nabla g_i(X_i, \zeta)\right\|\le 1/\rho=B\). Here, \(M = \max\{B^2b_i\}\) is a constant. The fourth equation is directly obtained from the triangle inequality. The last inequality again utilizes \(\left\|\widetilde\nabla_{R} g_i(X_i, \zeta)\right\| \leq \left\|\widetilde\nabla g_i(X_i, \zeta)\right\| \leq B\).

According to \Cref{thm:Riemannian subgradient inequality}, it holds that
\begin{align*}
   E_{\zeta_t\sim U(\Omega)
   }\left[f_\lambda(X^{t+1})\right]&\le f\left(P_{\lambda f}(X^t)\right)+\frac{2\mu_tc}{\lambda}\left[f\left(P_{\lambda f}(X^t),\zeta_t\right)-f\left(X^t,\zeta_t\right)\right]\\
   &\quad+\frac{(L+\tau)\mu_t}{\lambda}\|P_{\lambda f}(X^t)-X^t\|_F^2
  +\frac{KM^2}{\lambda}\mu_t^4+\frac{KB^2}{\lambda}\mu_t^2\\
   \end{align*}
 where \(c\) is a constant. According to the definitions of the Moreau envelope and the proximal mapping, the following holds:
\begin{align*}
 &f\left(P_{\lambda f}(X_t)\right)-f(X_t)+\frac{L+\tau}{2c}\|P_{\lambda f}(X_t)-X_t\|_F^2\\
 &=-\left[f(X_t)-f\left(P_{\lambda f}(X_t)\right)-\frac{L+\tau}{2c}\|P_{\lambda f}(X_t)-X_t\|_F^2\right]\\
 &=-\left[f(X_t)-\left(f\left(P_{\lambda f}(X_t)\right)+\frac{1}{2\lambda}\|P_{\lambda f}(X_t)-X_t\|_F^2\right)+\left(\frac{1}{2\lambda}-\frac{L+\tau}{2c}\right)\|P_{\lambda f}(X_t)-X_t\|_F^2\right]\\
 &\le -\left(\frac{1}{2\lambda}-\frac{L+\tau}{2c}\right)\|P_{\lambda f}(X_t)-X_t\|_F^2
\end{align*}

Thus, it can be deduced
\begin{align*}
E_{\zeta_t\sim U(\Omega)
   }\left[f_\lambda(X_{t+1})\right]&\le f_{\lambda} (X_{t})-\frac{2\mu_t c}{\lambda}\left(\frac{1}{2\lambda}-\frac{L+\tau}{2c}\right)\|P_{\lambda f}(X_t)-X_t\|_F^2+\frac{KM^2}{\lambda}\mu_t^4+\frac{KB^2}{\lambda}\mu_t^2
\end{align*}

After taking the expectation on both sides of the equation with respect to all previous realizations \( \left(\zeta_0, \cdots, \zeta_{t-1}\right)\), we obtain
\begin{align*}
E\left[f_\lambda(X_{t+1})\right]&\le E\left[f_{\lambda} (X_{t})\right]-\frac{2\mu_t c}{\lambda}\left(\frac{1}{2\lambda}-\frac{L+\tau}{2c}\right)E\left[\|P_{\lambda f}(X_t)-X_t\|_F^2\right]+\frac{KM^2}{\lambda}\mu_t^4+\frac{KB^2}{\lambda}\mu_t^2
\end{align*}
\end{proof}
Using  \Cref{eq:proposition6.3}, we can prove Theorem \ref{thm:convergence} and obtain our iteration complexity result for the block stochastic gradient descent (BSGD).
\begin{proof}[Proof of Theorem \ref{thm:convergence}]
By summing both sides of ~\eqref{eq:proposition6.3} from \(t=0,1,\ldots,T\), we obtain
$$
\begin{aligned}
\sum_{t=0}^T \mu_t E\left[\Theta^2(X_t)\right] \leq \frac{\big(f_{\lambda}(X_0) - \min f_{\lambda}\big) +\frac{KM^2}{\lambda}\sum_{t=0}^{T}\mu_t^4+\frac{KB^2}{\lambda}\sum_{t=0}^{T}\mu_t^2}{c\lambda \left(\frac{c}{\lambda} - \frac{L+\tau}{c}\right)}.
\end{aligned}
$$
From this, it is known
$$
\sum_{t=0}^T \frac{\mu_t}{\sum_{t=0}^T \mu_t} E\left[\Theta^2(X_t)\right] \leq \frac{1}{c\lambda \left(\frac{c}{\lambda} - \frac{L+\tau}{c}\right)} \frac{\big(f_{\lambda}(X_0) - \min f_{\lambda}\big) + \frac{KM^2}{\lambda}\sum_{t=0}^{T}\mu_t^4+\frac{KB^2}{\lambda}\sum_{t=0}^{T}\mu_t^2}{\sum_{t=0}^T \mu_t}.
$$
To complete the proof, the remaining task is to substitute \(\mu_t = \frac{1}{\sqrt{T+1}}\) into the above inequality and choose 
\begin{align}\label{eq:c1c2}
    c_1 = \frac{1}{c\lambda \left(\frac{c}{\lambda} - \frac{L+\tau}{c}\right)} \quad \text{and} \quad c_2 = \frac{K(M^2+B^2)}{\lambda}.
\end{align}
Note that the left-hand side is exactly \(E\left[\Theta^2\left(X_{\overline{t}}\right)\right]\), where the expectation is taken over \(\zeta_0, \cdots, \zeta_{T-1}, \overline{t}\).
\end{proof}

\end{document}